\documentclass[12pt,twoside]{amsart}
\usepackage{amsmath, amsthm, amscd, amsfonts, amssymb, graphicx}
\usepackage{enumerate}
\usepackage[colorlinks=true,
linkcolor=blue,
urlcolor=cyan,
citecolor=red]{hyperref}
\usepackage{mathrsfs}
\newtheorem{theorem}{Theorem}[section]
\newtheorem{lemma}{Lemma}[section]
\newtheorem{remark}{Remark}[section]

\newtheorem{corollary}{Corollary}[section]

\newtheorem{proposition}{Proposition}[section]
\numberwithin{equation}{section}

\begin{document}
\title{operator spectral geometric  versus geometric mean}
\author{Hamid Reza Moradi,  Shigeru Furuichi and Mohammad Sababheh}

\begin{abstract}
The main goal of this article is to present new inequalities for the spectral geometric mean $A\natural_t B$ of two positive definite operators $A, B$ on a Hilbert space. The obtained results complement many known inequalities for the geometric mean $A\sharp_t B$. In particular, explicit comparisons between 
 $A\natural_t B$ and $A\sharp_t B$ are given, with applications towards Ando-type inequalities and Ando-Hiai inequalities for $A\natural_t B$ and some other consequences.
\end{abstract}
\maketitle
\pagestyle{myheadings}
\markboth{\centerline {}}
{\centerline {}}
\bigskip
\bigskip

\section{Introduction}
Let $ \mathcal B\left( \mathcal H \right)$ be the $C^*$-algebra of all bounded linear operators on a Hilbert space $\mathcal{H}$. If $A\in \mathcal B\left( \mathcal H \right)$ is such that $\left<Ax,x\right>\geq 0$ for all $x\in\mathcal{H}$, $A$ is said to be positive semi-definite, and we write $A\geq 0$. If in addition $A$ is invertible, it is positive definite, and we write $A>0$. The class of positive definite operators in $\mathcal B\left( \mathcal H \right)$ will be denoted by $ \mathcal B\left( \mathcal H \right)^+$.
The ordering $A \ge B$  means $A-B\geq 0$. In addition, if $\alpha,\beta\in\mathbb{R}$, we write $\alpha \le A \le \beta$ if $\alpha I \le A\le \beta I$ where $I$ is the identity operator.

The weighted geometric mean of $A,B\in \mathcal B\left( \mathcal H \right)^+$ is defined by the equation
\begin{align}\label{eq_def_geo}
A\sharp_t B=A^{\frac{1}{2}}\left(A^{-\frac{1}{2}}BA^{-\frac{1}{2}}\right)^t A^{\frac{1}{2}}, \quad 0\leq t\leq 1.
\end{align}
When $t=\frac{1}{2}$, we simply write $A\sharp B$ instead of $A\sharp_{\frac{1}{2}}B.$
An interesting geometric meaning of $A\sharp B$ is that it is a midpoint of $A$ and $B$ for a natural Finsler metric (Thompson’s part metric) on the cone of positive definite operators \cite{Nus1, Nus2}.

The geometric mean $\sharp_t$ is a special case of operator means. Recall that an operator mean $\sigma$ on $ \mathcal B\left( \mathcal H \right)^+$ is a binary operation defined by
\begin{align*}
A\sigma B=A^{\frac{1}{2}}f\left(A^{-\frac{1}{2}}BA^{-\frac{1}{2}}\right)A^{\frac{1}{2}},
\end{align*}
where $f:(0,\infty)\to (0,\infty)$ is an operator monotone function, with $f(1)=1$.
Examples of other operator means are the weighted arithmetic and harmonic mean, defined respectively by
\begin{align*}
A\nabla_t B=(1-t)A+tB\;{\text{and}}\;A!_tB=\left((1-t)A^{-1}+tB^{-1}\right)^{-1}, 0\leq t\leq 1.
\end{align*}
It is well known that
\begin{align}\label{eq_means_comp}
A!_t B\leq A\sharp_t B\leq A\nabla_t B,\quad 0\leq t\leq 1.
\end{align}
In \cite{Fied1}, the spectral geometric mean was defined by 
\[A{{\natural}}B={{\left( {{A}^{-1}}\sharp B \right)}^{\frac{1}{2}}}A{\left( {{A}^{-1}}\sharp B \right)}^{\frac{1}{2}}.\]

Then the weighted spectral geometric mean was defined in \cite{Lee1} by
\[A{{\natural}_{t}}B={{\left( {{A}^{-1}}\sharp B \right)}^{t}}A{{\left( {{A}^{-1}}\sharp B \right)}^{t}}, 0\leq t\leq 1.\]

In fact, simple manipulations of the above definition implies  that $A{{\natural}_{t}}B$ is a unique positive definite solution $X$ of the following equation
\[{{\left( {{A}^{-1}}\sharp B \right)}^{t}}={{A}^{-1}}\sharp X.\]
We refer the reader to \cite{Kim1} as a recent reference for further algebraic and geometric meaning of $A\natural_t B$.

Our main target in this article is to present new inequalities for $A\natural_t B$, in a way that complements those inequalities known for $A\sharp_t B$. \\
In particular,\begin{itemize}
\item The relations between $\left<A\natural_t B x,x\right>$ and $\left<Ax,x\right>^{1-t}\left<Bx,x\right>^t$ are given.
\item Upper and lower bounds of $A\natural_t B$ in terms of $A\sharp_t B$ are given.
\item Ando-type inequalities are given for $\natural_t $, where the relation between $\Phi(A\natural_t B)$ and $\Phi(A)\natural_t \Phi(B)$ are described in both ways.
\item Kantorovich-type inequality is presented to find an upper bound of $\Phi\left(A^{-1}\right)\natural\Phi(A)^{-1}.$
\item An Ando-Hiai inequality for the spectral geometric mean is given to describe the relation between $\left(A\natural_t B\right)^r$ and $A^r\natural_t B^r$.
\item A detailed discussion of the Kantorovich constant is presented, as a by product of our results.
\end{itemize}

For this, we need to remind the reader of some well-established results for $A\sharp_t B.$ We also refer the reader to the obtained inequalities for spectral geometric mean that simulates the corresponding result for the geometric mean. This will help the reader better follow the order of the results.

\begin{lemma}\label{11}
\cite{10} Let $A,B\in   \mathcal B\left( \mathcal H \right)^+$, and let $\Phi $ be a unital positive linear map on $\mathcal B\left( \mathcal H \right)$. Then for any $0\le t\le 1$,
\[\Phi \left( A{{\sharp }_{t}}B \right)\le \Phi \left( A \right){{\sharp }_{t}}\Phi \left( B \right).\]
In particular, for any unit vector $x\in \mathcal H$,
\begin{equation}\label{eq_inner_geo}
\left\langle A{{\sharp}_{t}}Bx,x \right\rangle \le {{\left\langle Ax,x \right\rangle }^{1-t}}{{\left\langle Bx,x \right\rangle }^{t}}.
\end{equation}

\end{lemma}
We refer the reader to Theorems \ref{4} and \ref{thm_and_1} below for the spectral geometric mean versions of Lemma \ref{11}.

In this paper, we use the generalized Kantorovich constant for $0<m<M$ and $t\in \mathbb{R}$:
	\[K\left( m,M,t \right)=\frac{\left(m{{M}^{t}}-M{{m}^{t}}\right)}{\left( t-1 \right)\left( M-m \right)}{{\left( \frac{t-1}{t}\frac{{{M}^{t}}-{{m}^{t}}}{m{{M}^{t}}-M{{m}^{t}}} \right)}^{t}}.\]
	It is known that this  recovers the (original) Kantorovich constant when $t=-1$ and $t=2$, that is, $K(m,M,2)=K(m,M,-1)=\dfrac{(M+m)^2}{4Mm}$, whose constant was  appeared in  \cite[p.142]{Kantorovich} as the so-called Kantorovich inequality. In Section \ref{sec3}, we study the generalized Kantorovich constant is defined by
\begin{equation}\label{def_K01}
	K(x,t):=\frac{(x^t-x)}{(t-1)(x-1)}\left(\frac{t-1}{t}\frac{x^t-1}{x^t-x}\right)^t,\quad x>0,\,\,\,\,t\in\mathbb{R}.
\end{equation}
We see that $K(x,1,t)=K(x,t)$.
We use the same symbol $K$ as the generalized Kantorovich constant. However, this never creates confusion for the readers by the different numbers of their variables. 

\begin{lemma}\cite{2}\label{5}
Let $A,B\in \mathcal B\left( \mathcal H \right)$ be 	such that $mA\le B\le MA$ for some scalars $0<\text{}m\le M$, and let $\Phi $ be a unital positive linear map. Then for any $0\le t\le 1$,
	\[\Phi \left( A \right){{\sharp}_{t}}\Phi \left( B \right)\le \frac{1}{K\left( m,M,t \right)}\Phi \left( A{{\sharp }_{t}}B \right).\]
\end{lemma}
The spectral geometric mean version of Lemma \ref{5} is stated below in Theorem \ref{thm_rev_ando}.

\begin{lemma}\label{10}
\cite{1}
Let $A,B\in \mathcal B\left( \mathcal H \right)^+$ such that ${{m}_{1}}\le A\le {{M}_{1}}$, ${{m}_{2}}\le B\le {{M}_{2}}$ and let $\Phi $ be a unital positive linear map on $\mathcal B\left( \mathcal H \right)$. Then
\[\Phi \left( A \right)\sharp \Phi \left( B \right)\le \frac{\sqrt{{{M}_{1}}{{M}_{2}}}+\sqrt{{{m}_{1}}{{m}_{2}}}}{2\sqrt[4]{{{M}_{1}}{{m}_{1}}{{M}_{2}}{{m}_{2}}}}\Phi \left( A\sharp B \right).\]
In particular, for any unit vector $x\in \mathcal H$,
\[\sqrt{\left\langle Ax,x \right\rangle \left\langle Bx,x \right\rangle }\le \frac{\sqrt{{{M}_{1}}{{M}_{2}}}+\sqrt{{{m}_{1}}{{m}_{2}}}}{2\sqrt[4]{{{M}_{1}}{{m}_{1}}{{M}_{2}}{{m}_{2}}}}\left\langle A\sharp Bx,x \right\rangle.\]
\end{lemma}

The Choi-Davis inequality is stated next \cite{choi,davis}.

\begin{lemma}\label{lemm_choi}
 Let $f:J\to\mathbb{R}$ be an operator convex function, and let $\Phi:\mathcal{B}(\mathcal{H})\to \mathcal{B}(\mathcal{K})$ be a unital positive linear mapping. Then
\begin{equation}\label{choi_davis_ineq}
f(\Phi(A))\leq \Phi(f(A)),
\end{equation}
for all self-adjoint operators $A\in\mathcal{B}(\mathcal{H})$ with spectra in the interval $J$. \\
In particular, $\Phi(A)^{-1}\leq \Phi\left(A^{-1}\right),$ for any invertible self-adjoint operator $A$.

The inequality \eqref{choi_davis_ineq} is reversed when $f$ is operator concave.\\
In particular, if $0\leq t\leq 1,$ then
$$\Phi(A^t)\leq \Phi(A)^t,$$ for any positive definite operator $A$.

\end{lemma}

\begin{lemma}\label{12}
\cite{4} Let $A\in \mathcal B\left( \mathcal H \right)^+$ be such that $m\le A\le M$, and let $\Phi $ be a positive linear map on $\mathcal B\left( \mathcal H \right)$. Then for any $0\le t\le 1$,
\[\Phi {{\left( A \right)}^{t}}\le \frac{1}{K\left( m,M,t \right)}\Phi \left( {{A}^{t}} \right)\]
	In particular, for any unit vector $x\in \mathcal H$,
	\[{{\left\langle Ax,x \right\rangle }^{t}}\le \frac{1}{K\left( m,M,t \right)}\left\langle {{A}^{t}}x,x \right\rangle \]

\end{lemma}

\begin{lemma}\cite{6}\label{2}
Let $A\in \mathcal B\left( \mathcal H \right)$ such that $0<m\le A\le M$ and let $\Phi$ be a positive unital linear mapping. Then
$$\Phi\left(A^{-1}\right)\leq \frac{{{\left( M+m \right)}^{2}}}{4Mm} \Phi(A)^{-1}.$$ 

In particular, if $x\in \mathcal H$ is any unit vector, then
\[\left\langle {{A}^{-1}}x,x \right\rangle \le \frac{{{\left( M+m \right)}^{2}}}{4Mm}{{\left\langle Ax,x \right\rangle }^{-1}},\]
for any positive definite $A$.
\end{lemma}
We refer the reader to Corollary \ref{cor_kanto} below for the corresponding inequality that the spectral geometric mean fulfills, similar to Lemma \ref{2}.\\
While the inequality $A\leq B$ does not imply $A^2\leq B^2$, for arbitrary self-adjoint operators in general, the following useful weaker version holds.
\begin{lemma}\cite{3}\label{3}
Let $A,B\in \mathcal B\left( \mathcal H \right)$ be such that 	$0<A\le B$ and $m\le A\le M$. Then
\[{{A}^{2}}\le \frac{{{\left( M+m \right)}^{2}}}{4Mm}{{B}^{2}}.\]
\end{lemma}

\begin{lemma}\label{7}
\cite{5} Let $A,B\in \mathcal B\left( \mathcal H \right)$ be such that $m\le A,B\le M$ for some scalars $0\le m\le M$, and let $0\le r\le 1$. Then for any $0\le t\le 1$,
\[{{\left( A{{\nabla }_{t}}B \right)}^{r}}\le \frac{1}{K\left( m,M,t \right)}{{A}^{r}}{{\nabla }_{t}}{{B}^{r}}.\]
\end{lemma}

\begin{lemma}\label{6}
\cite{9} Let $A\in \mathcal B\left( \mathcal H \right)$ be a positive semi-definite operator. If $x\in \mathcal H$ is a unit vector, then
\begin{itemize}
\item[(i)] ${{\left\langle Ax,x \right\rangle }^{r}}\le \left\langle {{A}^{r}}x,x \right\rangle $ for all $r\ge 1$.
\item[(ii)] $\left\langle {{A}^{r}}x,x \right\rangle \le {{\left\langle Ax,x \right\rangle }^{r}}$ for all $0\le r\le 1$.
\item[(iii)] If A is invertible, then ${{\left\langle Ax,x \right\rangle }^{r}}\le \left\langle {{A}^{r}}x,x \right\rangle $ for all $r\le 0$.
\end{itemize}
\end{lemma}

We notice that the function $f(x)=x^t, 0\leq t\leq 1$ is operator monotone on $[0,\infty).$ This together with Lemma \ref{11} implies that 
\begin{equation}\label{needed_eq_11}
\Phi(A\sharp B)^t\leq \left(\Phi(A)\sharp\Phi(B)\right)^t, 0\leq t\leq 1
\end{equation}
for any positive definite operators $A,B.$
\section{Main results}
In this section, we present our main results, where we show some comparisons between the geometric and spectral geometric means first, then we discuss some Ando-type inequalities for the spectral geometric mean.

\subsection{Some comparisons between $A\sharp_tB$ and $A\natural_tB$}

The first result is the spectral geometric version of \eqref{eq_inner_geo}. For this, we notice first that if $m,M$ are positive scalars and $A,B\in\mathcal{B}(\mathcal{H})^+$ are such that $m\leq A,B\leq M$, then $M^{-1}\leq A^{-1}\leq m^{-1}$, which in turns imply 
\begin{equation}\label{needed_root}
\sqrt{\frac{m}{M}}\leq A^{-1}\sharp B\leq \sqrt{\frac{M}{m}},
\end{equation}
as a consequence of the monotone property of the geometric mean.
\begin{theorem}\label{4}
Let $A,B\in \mathcal B\left( \mathcal H \right)^+$  be such that $m\le A,B\le M$, and let $0\le t\le 1$. Then for any unit vector $x\in \mathcal H$,
\[\left\langle A{{\natural}_{t}}Bx,x \right\rangle \le \frac{{{\left( {{M}^{1+t}}+{{m}^{1+t}} \right)}^{2}}}{4{{M}^{1+t}}{{m}^{1+t}}}{{\left\langle Ax,x \right\rangle }^{1-t}}{{\left\langle Bx,x \right\rangle }^{t}}.\]
\end{theorem}
\begin{proof}
Using \eqref{needed_root} and by conjugation with $(A^{-1}\sharp B)^t$ we get
\[\frac{{{m}^{1+t}}}{{{M}^{t}}}\le X:={{\left( {{A}^{-1}}\sharp B \right)}^{t}}A{{\left( {{A}^{-1}}\sharp B \right)}^{t}}\le \frac{{{M}^{1+t}}}{{{m}^{t}}}.\]

On the other hand, we know that $A{{\natural}_{t}}B$ is a unique positive definite solution $X$ of the following equation
\begin{equation}\label{eq_for_spec}
{{\left( {{A}^{-1}}\sharp B \right)}^{t}}={{A}^{-1}}\sharp X.
\end{equation}
Let $X=A{{\natural}_{t}}B$. For any unit vector $x\in \mathcal H$,
\[\begin{aligned}
   \xi \sqrt{\left\langle {{A}^{-1}}x,x \right\rangle \left\langle Xx,x \right\rangle }&\le \left\langle {{A}^{-1}}\sharp Xx,x \right\rangle  \quad \text{(by Lemma \ref{10})}\\ 
 & =\left\langle {{\left( {{A}^{-1}}\sharp B \right)}^{t}}x,x \right\rangle  \\ 
 & \le {{\left\langle \left( {{A}^{-1}}\sharp B \right)x,x \right\rangle }^{t}} \quad \text{(by Lemma \ref{6})}\\ 
 & \le {{\left\langle {{A}^{-1}}x,x \right\rangle }^{\frac{t}{2}}}{{\left\langle Bx,x \right\rangle }^{\frac{t}{2}}} \quad \text{(by Lemma \ref{11})},
\end{aligned}\]
where $\xi =\dfrac{2\sqrt{{{M}^{1+t}}{{m}^{1+t}}}}{{{M}^{1+t}}+{{m}^{1+t}}}$.
Thus,
\[\left\langle Xx,x \right\rangle \le \frac{1}{{{\xi }^{2}}}{{\left\langle {{A}^{-1}}x,x \right\rangle }^{t-1}}{{\left\langle Bx,x \right\rangle }^{t}},\]
which implies,
\[\left\langle A{{\natural}_{t}}Bx,x \right\rangle \le \frac{1}{{{\xi }^{2}}}{{\left\langle {{A}^{-1}}x,x \right\rangle }^{t-1}}{{\left\langle Bx,x \right\rangle }^{t}}.\]
Again by Lemma \ref{6},
\[\left\langle A{{\natural}_{t}}Bx,x \right\rangle \le \frac{1}{{{\xi }^{2}}}{{\left\langle Ax,x \right\rangle }^{1-t}}{{\left\langle Bx,x \right\rangle }^{t}}.\]

\end{proof}

The relation between weighted spectral geometric mean, arithmetic mean, and harmonic mean is stated in the following result, in a way that simulates \eqref{eq_means_comp} for the geometric mean. This follows from Theorem \ref{4}, the weighted arithmetic-geometric mean, and the fact that ${{\left( {{A}^{-1}}{{\natural}_{t}}{{B}^{-1}} \right)}^{-1}}=A{{\natural}_{t}}B$.
\begin{corollary}\label{1}
Let $A,B\in \mathcal B\left( \mathcal H \right)^+$ be such that $0<m\le A,B\le M$, and let $0\le t\le 1$. Then
\[\frac{4{{M}^{1+t}}{{m}^{1+t}}}{{{\left( {{M}^{1+t}}+{{m}^{1+t}} \right)}^{2}}}A{{!}_{t}}B\le A{{\natural }_{t}}B\le \frac{{{\left( {{M}^{1+t}}+{{m}^{1+t}} \right)}^{2}}}{4{{M}^{1+t}}{{m}^{1+t}}}A{{\nabla }_{t}}B.\]
\end{corollary}

\begin{remark}
The inequalities in Corollary \ref{1} can be improved as follows. Recall the following refinement of the weighted arithmetic--geometric mean inequality \cite{8}
	\[{{\left( \frac{{{\left( 1+h \right)}^{2}}}{4h} \right)}^{r}}{{a}^{1-t}}{{b}^{t}}\le \left( \left( 1-t \right)a+tb \right)\]
where $h={b}/{a}\;$ and $r=\min \left\{ t,1-t \right\}$. Let $x$ be a unit vector. If we replace $a$ and $b$ by $\left\langle Ax,x \right\rangle $ and $\left\langle Bx,x \right\rangle $ with $m\le A\le m'<M'\le B\le M$ and  noting that the function \[f\left( h \right)=\frac{{{\left( 1+h \right)}^{2}}}{4h}\] 
is increasing for $h\ge 1,$ we get
	\[{{\left( \frac{{{\left( M'+m' \right)}^{2}}}{4M'm'} \right)}^{r}}{{\left\langle Ax,x \right\rangle }^{1-t}}{{\left\langle Bx,x \right\rangle }^{t}}\le \left\langle \left( A{{\nabla }_{t}}B \right)x,x \right\rangle.\]
Thus,
	\[A{{\natural }_{t}}B\le \frac{{{\left( {{M}^{1+t}}+{{m}^{1+t}} \right)}^{2}}}{4{{M}^{1+t}}{{m}^{1+t}}}{{\left( \frac{4M'm'}{{{\left( M'+m' \right)}^{2}}} \right)}^{r}}A{{\nabla }_{t}}B,\]
whenever $m\le A\le m'<M'\le B\le M$. The same inequality holds when $m\le B\le m'<M'\le A\le M$.
\end{remark}

The following result intends to give a relationship between the weighted spectral geometric mean and the weighted geometric mean.
\begin{corollary}\label{cor_spec_less_geo}
Let $A,B\in \mathcal B\left( \mathcal H \right)$ be such that $0<m\le A,B\le M$, and let $0\le t\le 1$. Then
\[A{{\natural }_{t}}B\le \frac{1}{K\left( \frac{m}{M},\frac{M}{m},t \right)}\frac{{{\left( {{M}^{1+t}}+{{m}^{1+t}} \right)}^{2}}}{4{{M}^{1+t}}{{m}^{1+t}}}A{{\sharp }_{t}}B.\]
\end{corollary}
\begin{proof}
If $m\le A,B\le M$, then $\frac{m}{M}A\le B\le \frac{M}{m}A$. In this case, by Lemma \ref{5}, we get
\[\Phi \left( A \right){{\sharp}_{t}}\Phi \left( B \right)\le \frac{1}{K\left( \frac{m}{M},\frac{M}{m},t \right)}\Phi \left( A{{\sharp }_{t}}B \right),\]
for any positive unital linear mapping $\Phi.$
In particular,
\[{{\left\langle Ax,x \right\rangle }^{1-t}}{{\left\langle Bx,x \right\rangle }^{t}}\le \frac{1}{K\left( \frac{m}{M},\frac{M}{m},t \right)}\left\langle A{{\sharp}_{t}}Bx,x \right\rangle\]
for any unit vector $x\in \mathcal H$.
Now, by Theorem \ref{4}, we have
\[A{{\natural }_{t}}B\le \frac{1}{K\left( \frac{m}{M},\frac{M}{m},t \right)}\frac{{{\left( {{M}^{1+t}}+{{m}^{1+t}} \right)}^{2}}}{4{{M}^{1+t}}{{m}^{1+t}}}A{{\sharp }_{t}}B.\]
\end{proof}

A reverse of Theorem \ref{4} can be stated as follows.
\begin{theorem}
Let $A,B\in \mathcal B\left( \mathcal H \right)$ be such that $0<m\le A,B\le M$, and let $0\le t\le 1$. Then for any unit vector $x\in \mathcal H$,
\[{{\left\langle Ax,x \right\rangle }^{1-t}}{{\left\langle Bx,x \right\rangle }^{t}}\le \frac{1}{{{K}^{2}}\left( \sqrt{\frac{m}{M}},\sqrt{\frac{M}{m}},t \right)}{{ \frac{{{\left( M+m \right)}^{2}}}{4Mm}}}\left\langle A{{\natural }_{t}}Bx,x \right\rangle,\]
\end{theorem}
\begin{proof}
We have
\[\begin{aligned}
   \sqrt{\left\langle {{A}^{-1}}x,x \right\rangle \left\langle Xx,x \right\rangle }&\ge \left\langle {{A}^{-1}}\sharp Xx,x \right\rangle  \quad \text{(by Lemma \ref{11})}\\ 
 & =\left\langle {{\left( {{A}^{-1}}\sharp B \right)}^{t}}x,x \right\rangle  \\ 
 & \ge K\left( \sqrt{\frac{m}{M}},\sqrt{\frac{M}{m}},t \right){{\left\langle {{A}^{-1}}\sharp Bx,x \right\rangle }^{t}} \quad \text{(by Lemma \ref{12})}\\ 
 & \ge K\left( \sqrt{\frac{m}{M}},\sqrt{\frac{M}{m}},t \right){{\left( \frac{2\sqrt{Mm}}{M+m} \right)}^{t}}{{\left\langle {{A}^{-1}}x,x \right\rangle }^{\frac{t}{2}}}{{\left\langle Bx,x \right\rangle }^{\frac{t}{2}}}\quad \text{(by Lemma \ref{10})}.  
\end{aligned}\]
Thus,
\[\left\langle A{{\natural}_{t}}Bx,x \right\rangle \ge {{K}^{2}}\left( \sqrt{\frac{m}{M}},\sqrt{\frac{M}{m}},t \right){{\left( \frac{4Mm}{{{\left( M+m \right)}^{2}}} \right)}^{t}}{{\left\langle {{A}^{-1}}x,x \right\rangle }^{t-1}}{{\left\langle Bx,x \right\rangle }^{t}}.\]
Now, by Lemma \ref{2}, we have
\[{{\left\langle Ax,x \right\rangle }^{1-t}}{{\left\langle Bx,x \right\rangle }^{t}}\le \frac{1}{{{K}^{2}}\left( \sqrt{\frac{m}{M}},\sqrt{\frac{M}{m}},t \right)}{{ \frac{{{\left( M+m \right)}^{2}}}{4Mm} }}\left\langle A{{\natural }_{t}}Bx,x \right\rangle,\]
as desired.
\end{proof}

Now we are ready to present a reversed version of Corollary \eqref{cor_spec_less_geo}.
\begin{corollary}\label{cor_geo_less_spec}
Let $A,B\in \mathcal B\left( \mathcal H \right)$ be such that $0<m\le A,B\le M$, and let $0\le t\le 1$. Then
\[A{{\sharp}_{t}}B\le \frac{1}{{{K}^{2}}\left( \sqrt{\frac{m}{M}},\sqrt{\frac{M}{m}},t \right)}{{ \frac{{{\left( M+m \right)}^{2}}}{4Mm} }}A{{\natural}_{t}}B.\]
\end{corollary}

A detailed discussion of the relation between Corollaries \ref{cor_spec_less_geo} and \ref{cor_geo_less_spec} is given in Proposition \ref{last_prop}. The proposition is stated later in this paper due to the detailed computations of the generalized Kantorovich constant, which we prefer to state independently to avoid disturbing the readers.

\subsection{Ando-type inequalities for $A\natural_t B$}

We begin this section by presenting the following Ando-type inequality, similar to the celebrated result given in Lemma \ref{11}.

\begin{theorem}\label{thm_and_1}
Let $A,B\in \mathcal B\left( \mathcal H \right)$ be positive definite such that $0<m\leq A,B\leq M$ and let $\Phi$ be a positive unital linear mapping. Then
$$\Phi(A\natural_t B)\leq \beta \Phi(A)\natural_t \Phi(B), 0\leq t\leq 1,$$
where $\beta =\dfrac{{{\left( M+m \right)}^{2t}}{{\left( {{M}^{1+t}}+{{m}^{1+t}} \right)}^{4}}}{{{4}^{2+t}}{{M}^{2+3t}}{{m}^{2+3t}}}$.
\end{theorem}
\begin{proof}
Let $X=A\natural_t B$ and $Y=\Phi(A)\natural_t \Phi(B).$ Then
\begin{align*}
\Phi\left(A^{-1}\sharp X\right)&=\Phi\left(\left(A^{-1}\sharp B\right)^t\right)\;\;({\text{by}}\;\eqref{eq_for_spec})\\
&\leq \Phi\left(A^{-1}\sharp B\right)^t\;\;\;({\text{by\;Lemma}}\;\ref{lemm_choi})\\
&\leq \left(\Phi\left(A^{-1}\right)\sharp \Phi(B)\right)^t\;\;({\text{by}}\;\eqref{needed_eq_11})\\
&\leq \left(C_1 \Phi(A)^{-1}\sharp \Phi(B)\right)^t\;\;({\text{by\;Lemma}}\;\ref{2})\\
&=C_1^{\frac{t}{2}}\left(\Phi(A)^{-1}\sharp \Phi(B)\right)^t\\
&=C_1^{\frac{t}{2}} \Phi(A)^{-1}\sharp Y,
\end{align*}
where ${{C}_{1}}=\dfrac{{{\left( M+m \right)}^{2}}}{4Mm}$.  On the other hand, using the reversed Ando's inequality, we obtain
\begin{align*}
\Phi\left(A^{-1}\sharp X\right)&\geq C_2 \Phi(A^{-1})\sharp \Phi(X)\;\;({\text{by\;Lemma}}\;\ref{10})\\
&\geq C_2 \Phi(A)^{-1}\sharp \Phi(X)\;\;({\text{by\;Lemma}}\;\ref{lemm_choi})
\end{align*}
where ${{C}_{2}}=\dfrac{2\sqrt{{{M}^{1+t}}{{m}^{1+t}}}}{{{M}^{1+t}}+{{m}^{1+t}}}$. Thus, we have shown that
\begin{align*}
C_2 \Phi(A)^{-1}\sharp \Phi(X)\leq C_1^{\frac{t}{2}} \Phi(A)^{-1}\sharp Y.
\end{align*}
This leads to
\begin{align*}
C_2\left( \Phi(A)^{\frac{1}{2}}\Phi(X)\Phi(A)^{\frac{1}{2}}         \right)^{\frac{1}{2}}\leq C_1^{\frac{t}{2}}\left(   \Phi(A)^{\frac{1}{2}} Y\Phi(A)^{\frac{1}{2}}                 \right)^{\frac{1}{2}}.
\end{align*}
Now, by Lemma \ref{3}, we infer that
\begin{align*}
C_2^2\left( \Phi(A)^{\frac{1}{2}}\Phi(X)\Phi(A)^{\frac{1}{2}}         \right)\leq C_1^tK_2\left(   \Phi(A)^{\frac{1}{2}} Y\Phi(A)^{\frac{1}{2}}                 \right),
\end{align*}
where ${{K}_{2}}=\dfrac{{{\left( {{M}^{1+t}}+{{m}^{1+t}} \right)}^{2}}}{4{{M}^{1+t}}{{m}^{1+t}}}$, which in turns implies
\begin{align*}
\Phi(A\natural_t B)\leq \frac{C_1^tK_2}{C_2^2}\Phi(A)\natural_t \Phi(B).
\end{align*}
This completes the proof.
\end{proof}

Having shown Theorem \ref{thm_and_1} as an Ando-type inequality for the spectral geometric mean, the following theorem intends to give a reversed version of Theorem \ref{thm_and_1}.

\begin{theorem}\label{thm_rev_ando}
Let $A,B\in \mathcal B\left( \mathcal H \right)$ be positive definite such that $0<m\leq A,B\leq M$ and let $\Phi$ be a positive unital linear mapping. Then
\[\Phi \left( A \right){{\natural}_{t}}\Phi \left( B \right)\le \frac{1}{{{K}^{2}}\left( \frac{m}{M},\frac{M}{m},t \right)}{{\left( \frac{M+m}{2\sqrt{Mm}} \right)}^{2t}}\frac{{{\left( {{M}^{1+t}}+{{m}^{1+t}} \right)}^{2}}}{4{{M}^{1+t}}{{m}^{1+t}}}\Phi \left( A{{\natural }_{t}}B \right).\]
\end{theorem}
\begin{proof}
Let $X=A\natural_t B$. Then
\[\begin{aligned}
   \Phi \left( {{A}^{-1}}\sharp X \right)&=\Phi \left( {{\left( {{A}^{-1}}\sharp B \right)}^{t}} \right)\quad({\text{by}}\;\eqref{eq_for_spec}) \\ 
 & \ge K\left( \frac{m}{M},\frac{M}{m},t \right)\Phi {{\left( {{A}^{-1}}\sharp B \right)}^{t}}\quad({\text{by\;Lemma}}\;\ref{12}) \\ 
 & \ge K\left( \frac{m}{M},\frac{M}{m},t \right){{\left( \frac{2\sqrt{Mm}}{M+m}\Phi \left( {{A}^{-1}} \right)\sharp \Phi \left( B \right) \right)}^{t}}\quad({\text{by\;Lemma}}\;\ref{5}) \\ 
 & =K\left( \frac{m}{M},\frac{M}{m},t \right){{\left( \frac{2\sqrt{Mm}}{M+m} \right)}^{t}}{{\left( \Phi \left( {{A}^{-1}} \right)\sharp \Phi \left( B \right) \right)}^{t}} \\ 
 & \ge K\left( \frac{m}{M},\frac{M}{m},t \right){{\left( \frac{2\sqrt{Mm}}{M+m} \right)}^{t}}{{\left( \Phi {{\left( A \right)}^{-1}}\sharp \Phi \left( B \right) \right)}^{t}}\;({\text{by\;Lemma}}\;\ref{lemm_choi}) \\ 
 & =K\left( \frac{m}{M},\frac{M}{m},t \right){{\left( \frac{2\sqrt{Mm}}{M+m} \right)}^{t}}\Phi {{\left( A \right)}^{-1}}\sharp Y,
\end{aligned}\]
where $Y=\Phi(A)\natural_t \Phi(B)$; see \eqref{eq_for_spec}.  This implies
\[\Phi {{\left( A \right)}^{-1}}\sharp Y\le \frac{1}{K\left( \frac{m}{M},\frac{M}{m},t \right)}{{\left( \frac{M+m}{2\sqrt{Mm}} \right)}^{t}}\Phi \left( {{A}^{-1}} \right)\sharp \Phi \left( X \right).\]
Therefore,
\[{{\left( \Phi {{\left( A \right)}^{\frac{1}{2}}}Y\Phi {{\left( A \right)}^{\frac{1}{2}}} \right)}^{\frac{1}{2}}}\le \frac{1}{K\left( \frac{m}{M},\frac{M}{m},t \right)}{{\left( \frac{M+m}{2\sqrt{Mm}} \right)}^{t}}{{\left( \Phi {{\left( A \right)}^{\frac{1}{2}}}\Phi \left( X \right)\Phi {{\left( A \right)}^{\frac{1}{2}}} \right)}^{\frac{1}{2}}}.\]
By Lemma \ref{3},
\[\Phi \left( A \right){{\natural}_{t}}\Phi \left( B \right)\le \frac{1}{{{K}^{2}}\left( \frac{m}{M},\frac{M}{m},t \right)}{{\left( \frac{M+m}{2\sqrt{Mm}} \right)}^{2t}}\frac{{{\left( {{M}^{1+t}}+{{m}^{1+t}} \right)}^{2}}}{4{{M}^{1+t}}{{m}^{1+t}}}\Phi \left( A{{\natural }_{t}}B \right),\]
as desired.
\end{proof}
An upper bound of $\Phi(A)\sharp_t \Phi(A)^{-1}$ in terms of the Kantorovich constant is usually stated as the Kantorovich inequality. In the following, we present this inequality for the spectral geometric mean.
\begin{corollary}\label{cor_kanto}
(Operator Kantorovich inequality for spectral geometric mean) Let $A\in \mathcal B\left( \mathcal H \right)$ be positive definite such that $0<m\leq A\leq M$ and let $\Phi$ be a positive unital linear mapping. If $0\leq t\leq 1,$ then
\[\Phi \left( A \right)\natural_t \Phi \left( {{A}^{-1}} \right)\le \frac{1}{{{K}^{2}}\left( \frac{1}{M},\frac{1}{m},t \right)}\left( \frac{M+m}{2\sqrt{Mm}} \right)\frac{{{\left( {{{M}^{1+t}}}+{{{m}^{1+t}}} \right)}^{2}}}{4{{{M}^{1+t}}{{m}^{1+t}}}}.\]
\end{corollary}

\begin{corollary}\label{8}
Let $A,B\in \mathcal B\left( \mathcal H \right)$ be such that $0<m\le A,B\le M$, and let $0\le t\le 1$. Then
\[A{{\nabla }_{t}}B\le \frac{m{{\nabla }_{\lambda }}M}{m{{\sharp }_{\lambda }}M}\frac{1}{{{K}^{2}}\left( \sqrt{\frac{m}{M}},\sqrt{\frac{M}{m}},t \right)}{{\left( \frac{{{\left( M+m \right)}^{2}}}{4Mm} \right)}^{1+t}}A{{\natural }_{t}}B,\]
where $\lambda =\min \left\{ t,1-t \right\}$.
\end{corollary}
\begin{proof}
It has been shown in \cite{7} that
\begin{equation}\label{9}
A{{\nabla }_{t}}B\le \frac{m{{\nabla }_{\lambda }}M}{m{{\sharp }_{\lambda }}M}A{{\sharp }_{t}}B.
\end{equation}
Now, the result follows by combining \eqref{9} and Corollary \ref{cor_geo_less_spec}.
\end{proof}

\begin{theorem}\label{AH_SGM_theorem01}
(Ando-Hiai inequality for spectral geometric mean)  Let $A,B\in \mathcal B\left( \mathcal H \right)$ be such that $0<m\le A,B\le M$, and let $0\le r\le 1$. Then
\[{{\left( A{{\natural }_{t}}B \right)}^{r}}\le \min\{\kappa_1(m,M,r,t,\lambda),\kappa_2(m,M,r,t,\lambda)\}{{A}^{r}}{{\natural }_{t}}{{B}^{r}},\]
where 
\begin{eqnarray*}
&&\kappa_1(m,M,r,t,\lambda):=\frac{\frac{m^r\nabla_{\lambda}M^r}{m^r\sharp_{\lambda}M^r}\left(\frac{(M^r+m^r)^2}{4M^rm^r}\right)^{1+t}}{K(m,M,t)K^2\left(\sqrt{\frac{m^r}{M^r}},\sqrt{\frac{M^r}{m^r}},t\right)}\frac{\left(\frac{(M+m)^2}{4Mm}\right)^{r(1+t)}}{K^{2r}\left(\sqrt{m}{M},\sqrt{M}{m},t\right)},\\
&&\kappa_2(m,M,r,t,\lambda):=\frac{\frac{m^r\nabla_{\lambda}M^r}{m^r\sharp_{\lambda}M^r}\left(\frac{(M^r+m^r)^2}{4M^rm^r}\right)^{1+t}}{K(m,M,t)K^2\left(\sqrt{\frac{m^r}{M^r}},\sqrt{\frac{M^r}{m^r}},t\right)}\left(\frac{(M^{1+t}+m^{1+t})^2}{4M^{1+t}m^{1+t}}\right)^r,
\end{eqnarray*}
for  $0\le t\le 1$ and $\lambda =\min \left\{ t,1-t \right\}$.
\end{theorem}
\begin{proof}
In Corollary \ref{cor_geo_less_spec}, we replace $A$ and $B$ by $A^{-1}$ and $B^{-1}$, respectively. Then we have
\begin{equation}\label{AH_SGM_theorem_ineq01}
A{{\natural}_{t}}B\le \frac{1}{{{K}^{2}}\left( \sqrt{\frac{m}{M}},\sqrt{\frac{M}{m}},t \right)}{{\left( \frac{{{\left( M+m \right)}^{2}}}{4Mm} \right)}^{1+t}}A{{\sharp}_{t}}B,
\end{equation}
since $\left(A^{-1}{{\natural}_{t}}B^{-1}\right)^{-1}=A{{\natural}_{t}}B$, $\left(A^{-1}{{\sharp}_{t}}B^{-1}\right)^{-1}=A{{\sharp}_{t}}B$ and
$$
K^{-2}\left(\sqrt{\frac{1/M}{1/m}},\sqrt{\frac{1/m}{1/M}},t\right)\left(\frac{(1/m+1/M)^2}{4/mM}\right)^{1+t}=K^{-2}\left(\sqrt{\frac{m}{M}},\sqrt{\frac{M}{m}},t\right)\left(\frac{(M+m)^2}{4mM}\right)^{1+t}.
$$
Since the function $f\left( x \right)={{x}^{r}}\left( 0\le r\le 1 \right)$ is  operator monotone, we have
\[\begin{aligned}
  & {{K}^{2r}}\left( \sqrt{\frac{m}{M}},\sqrt{\frac{M}{m}},t \right){{\left( \frac{4Mm}{{{\left( M+m \right)}^{2}}} \right)}^{r\left( 1+t \right)}}{{\left( A{{\natural }_{t}}B \right)}^{r}} \\ 
 & \le {{\left( A{{\sharp }_{t}}B \right)}^{r}} \quad \text{(by \eqref{AH_SGM_theorem_ineq01})}\\ 
 & \le {{\left( A{{\nabla }_{t}}B \right)}^{r}} \quad \text{(by the weighted arithmetic--geometric operator mean inequality)}\\ 
 & \le \frac{1}{K\left( m,M,t \right)}{{A}^{r}}{{\nabla }_{t}}{{B}^{r}} \quad \text{(by Lemma \ref{7})}\\ 
 & \le \frac{\frac{{{m}^{r}}{{\nabla }_{\lambda }}{{M}^{r}}}{{{m}^{r}}{{\sharp }_{\lambda }}{{M}^{r}}}}{K\left( m,M,t \right){{K}^{2}}\left( \sqrt{\frac{{{m}^{r}}}{{{M}^{r}}}},\sqrt{\frac{{{M}^{r}}}{{{m}^{r}}}},t \right)}{{\left( \frac{{{\left( {{M}^{r}}+{{m}^{r}} \right)}^{2}}}{4{{M}^{r}}{{m}^{r}}} \right)}^{1+t}}\left( {{A}^{r}}{{\natural}_{t}}{{B}^{r}} \right) \quad \text{(by Corollary \ref{8})}.\\ 
\end{aligned}\]
Consequently,
\[{{\left( A{{\natural }_{t}}B \right)}^{r}}\le \frac{\frac{{{m}^{r}}{{\nabla }_{\lambda }}{{M}^{r}}}{{{m}^{r}}{{\sharp }_{\lambda }}{{M}^{r}}}{{\left( \frac{{{\left( M+m \right)}^{2}}}{4Mm} \right)}^{r\left( 1+t \right)}}{{\left( \frac{{{\left( {{M}^{r}}+{{m}^{r}} \right)}^{2}}}{4{{M}^{r}}{{m}^{r}}} \right)}^{1+t}}}{K\left( m,M,t \right){{K}^{2r}}\left( \sqrt{\frac{m}{M}},\sqrt{\frac{M}{m}},t \right){{K}^{2}}\left( \sqrt{\frac{{{m}^{r}}}{{{M}^{r}}}},\sqrt{\frac{{{M}^{r}}}{{{m}^{r}}}},t \right)}{{A}^{r}}{{\natural }_{t}}{{B}^{r}}.\]

Similarly we have,
\begin{eqnarray*}
&& \left(A\natural_tB\right)^r\le \left(\frac{(M^{1+t}+m^{1+t})^2}{4M^{1+t}m^{1+t}}\right)^r\left(A\nabla_t B\right)^r\quad \text{(by Corollary \ref{1})}\\
&& \le  \left(\frac{(M^{1+t}+m^{1+t})^2}{4M^{1+t}m^{1+t}}\right)^r\frac{1}{K(m,M,t)} A^r \nabla_t B^r\quad \text{(by Lemma \ref{7})}\\
&& \le \frac{m^r\nabla M^r}{m^r\sharp M^r}\frac{\left(\frac{(M^{1+t}+m^{1+t})^2}{4M^{1+t}m^{1+t}}\right)^r\left(\frac{(M^r+m^r)^2}{4M^rm^r}\right)^{1+t}}{K(m,M,t)K^2\left(\sqrt{\frac{m^r}{M^r}},\sqrt{\frac{M^r}{m^r}},t\right)} A^r \natural_t B^r\quad \text{(by Corollary \ref{8})}.
\end{eqnarray*}
\end{proof}

\begin{remark}
There is no ordering of two constants $\kappa_1(m,M,r,t,\lambda)$ and $\kappa_2(m,M,r,t,\lambda)$ appearing in Theorem \ref{AH_SGM_theorem01}.
To compare them, it suffices to consider the function
$$
\delta(x,t):=\left(\frac{(x+1)^2}{4x}\right)^{t+1}-\frac{(x^{t+1}+1)^2}{4x^{t+1}}K^2(x,t),\quad (x>1,\,\,\,0\le t \le 1),
$$
by putting $x:=M/m>1$.
Then we have $\delta(10,0.1)\simeq 0.10068$ and $\delta(10,0.9)\simeq -10.011$.
Although we generally have the inequality $\left(\dfrac{(x+1)^2}{4x}\right)^{t+1} \le \dfrac{(x^{t+1}+1)^2}{4x^{t+1}}$ for $x>1$ and $0\le t \le 1$ by elementary calculations, the fact $K^2(x,t)\le 1$ for $0\le t \le 1$, affects this comparison.
\end{remark}

\section{On the generalized Kantorovich constant}\label{sec3}
In our discussion of the above inequalities, different constants have shown up. This section attempts to give a detailed discussion of these constants. This helps better understand the stated results. The conclusion of this section is a hidden relation between Corollaries \ref{cor_spec_less_geo} and \ref{cor_geo_less_spec}.	
	
    Due to the technical nature of the proofs of the results in this section, we have added the proofs in the appendix at the end of this paper. So, results in Section 3 will be presented without proofs so that the reader can easily follow without distraction.

In the following proposition, we present a monotonicity behavior of the generalized Kantorovich constant $K(x,t)$.
\begin{proposition}\label{proposition01}
The generalized Kantorovich constant $K(x,t)$ satisfies the following monotonicity properties in $x$. 
\begin{itemize}
\item[(i)] If $t<0$ or $t>1$, then $K(x,t)$ is monotone decreasing when $0<x<1$, and monotone increasing when $x>1$. $K(x,t)$ takes the minimum value $1$ when $x=1$.
\item[(ii)] If $0<t<1$, then $K(x,t)$ is monotone increasing when $0<x<1$, and monotone decreasing when $x>1$. $K(x,t)$ takes the maximum value $1$ when $x=1$.
\end{itemize}
\end{proposition}

\begin{remark}
The case  $0<x<1$ in Proposition \ref{proposition01} can be proven by the fact $K(1/x,t)=K(x,t)$ for any $t >0$ and $x>0$,  \cite[Theorem 2.54 (i)]{FHPS2005}.
\end{remark}
 
 Proposition \ref{proposition01} can be stated equivalently in the following form, where
\begin{equation}\label{def_K02}
K(m,M,t):=\frac{(mM^t-Mm^t)}{(t-1)(M-m)}\left(\frac{t-1}{t}\frac{M^t-m^t}{mM^t-Mm^t}\right)^t,\quad (0<m<M)
\end{equation}

\begin{corollary}\label{corollary01}
Let $0<m_1<M_1$ and $0<m_2<M_2$ such that $\dfrac{M_1}{m_1}\le \dfrac{M_2}{m_2}$.
If $0<t<1$, then $K(m_1,M_1,t) \ge K(m_2,M_2,t)$. If $t<0$ or $t>1$, then $K(m_1,M_1,t) \le K(m_2,M_2,t)$.
\end{corollary}

To better understand these constants, and to reach our goal, we need the following lemma.
\begin{lemma}\label{lemma01}
\hfill
\begin{itemize}
\item[(I)] Let $0< t < 1$.  For $x > 1$, the following inequality holds
\begin{equation}\label{lemma01_ineq01}
\frac{x+1}{x-1}> (1-t)^2\left(\frac{x+x^t}{x-x^t}\right)+t^2\left(\frac{x^t+1}{x^t-1}\right).
\end{equation}
For $0<x<1$, we have the inequality \eqref{lemma01_ineq01r} below.

\item[(II)] Let $t > 1$ or $t < 0$. For $x>1$, the following inequality holds
\begin{equation}\label{lemma01_ineq01r}
\frac{x+1}{x-1}< (1-t)^2\left(\frac{x+x^t}{x-x^t}\right)+t^2\left(\frac{x^t+1}{x^t-1}\right).
\end{equation}
For $0<x<1$, we have the inequality \eqref{lemma01_ineq01} above.
\end{itemize}
\end{lemma}

Applying Lemma \ref{lemma01}, we have the following bounds of Kantorovich constant $K(x,t)$.
\begin{theorem}\label{theorem01}
Define $L(x,t):=\left(\dfrac{x^t+x}{x+1}\right)\left(\dfrac{x^t+1}{x^t+x}\right)^t$ for $x>0$ and $t \in \mathbb{R}$. 
\begin{itemize}
\item[(I)] If $0\le t \le 1$, then 
\begin{equation}\label{theorem01_ineq01}
L(x,t)\le K(x,t),\quad (x>0).
\end{equation}
\item[(II)] If $t\ge 1$ or $t\le 0$, then
\begin{equation}\label{theorem01_ineq01r}
L(x,t)\ge K(x,t),\quad (x>0).
\end{equation}
\end{itemize}
\end{theorem}

Note that we have 
$$
L(x,1/2)=\frac{\sqrt{x}(\sqrt{x}+1)^2}{x+1}\le \frac{2\sqrt[4]{x}}{\sqrt{x}+1}=K(x,1/2)
$$
and
$$
L(x,2)=L(x,-1)=\frac{(x^2+1)^2}{x(x+1)^2}\ge \frac{(x+1)^2}{4x}=K(x,2)=K(x,-1)
$$
for special cases.

\begin{corollary}\label{cor01}
\hfill
\begin{itemize}
\item[(I)] For $0\le t \le 1$, we have
\begin{equation}\label{corollary01_ineq01}
K(x^2,t)\le K^2(x,t),\quad (x>0).
\end{equation}
\item[(II)] For $t\ge 1$ or $t \le 0$, we have
\begin{equation}\label{corollary02_ineq01}
K(x^2,t)\ge K^2(x,t),\quad (x>0).
\end{equation}
\end{itemize}
\end{corollary}


Now we are ready to present the relation between Corollaries \ref{cor_spec_less_geo} and \ref{cor_geo_less_spec}.
\begin{proposition}\label{last_prop}
Let $A,B\in \mathcal B\left( \mathcal H \right)$ be such that $0<m\le A,B\le M$, and let $0\le t\le 1$. Then
\begin{equation}\label{last_prop_ineq01}
A{{\natural }_{t}}B\le \eta(m,M,t)A{{\sharp }_{t}}B\le \Gamma(m,M,t)A{{\sharp }_{t}}B,
\end{equation}
where $\eta(m,M,t)$ and $\Gamma(m,M,t)$ are shown in Corollary \ref{cor_geo_less_spec} and Corollary \ref{cor_spec_less_geo}, respectively in the following:
$$
\eta(m,M,t):=\frac{1}{{{K}^{2}}\left( \sqrt{\frac{m}{M}},\sqrt{\frac{M}{m}},t \right)}{{\left( \frac{{{\left( M+m \right)}^{2}}}{4Mm} \right)}^{1+t}},\quad \Gamma(m,M,t):=\frac{1}{K\left( \frac{m}{M},\frac{M}{m},t \right)}\frac{{{\left( {{M}^{1+t}}+{{m}^{1+t}} \right)}^{2}}}{4{{M}^{1+t}}{{m}^{1+t}}}.
$$
\end{proposition}
Note that we have $A\natural_tB\le\eta(m,M,t)A\sharp_tB$, replacing $A$ and $B$ by $A^{-1}$ and $B^{-1}$ respectively in Corollary \ref{cor_geo_less_spec}, with $\eta(m,M,t)=\eta\left(\frac{1}{M},\frac{1}{m},t\right)$.
From Proposition \ref{last_prop}, we can claim that Corollary \ref{cor_geo_less_spec} gives the tighter bound than Corollary \ref{cor_spec_less_geo}.

The following is the final remark on our bound $L(x,t)$ appeared in Theorem \ref{theorem01}, for the generalized Kantorovich constant $K(x,t)$.

\begin{remark}
It is known \cite[Theorem 2.54 (iv)]{FHPS2005} that $K(x,t)$ is decreasing for $t<1/2$ and increasing for $t>1/2$. Therefore $K(x,t)$ takes the minimum value $\dfrac{2x^{1/4}}{\sqrt{x}+1}$ when $t=1/2$.
It is natural to ask whether the following inequality holds or not:
$$
L(x,t) \le \dfrac{2x^{1/4}}{\sqrt{x}+1},\quad (x>0,\,\,\,0\le t \le 1).
$$
However, we have the following example.
$$
\dfrac{2x^{1/4}}{\sqrt{x}+1} -L(x,t) \simeq -0.0171811,\quad (x=10,\,\,\,\,t=0.1).
$$
\end{remark}

\section*{Acknowledgement}
The authors would like to thank the referees for their careful and insightful suggestions to improve our manuscript. 
The author (S.F.) was partially supported by JSPS KAKENHI Grant Number 21K03341.

\vskip 0.3 true cm

{\tiny (H. R. Moradi) Department of Mathematics, Payame Noor University (PNU), P.O.Box, 19395-4697, Tehran, Iran.
	
	\textit{E-mail address:} hrmoradi@mshdiau.ac.ir}

\vskip 0.3 true cm 	 

{\tiny (S. Furuichi) Department of Information Science, College of Humanities and Sciences, Nihon University,	Setagaya-ku, Tokyo, 156-8550, Japan. }

{\tiny \textit{E-mail address:} furuichi.shigeru@nihon-u.ac.jp }

\vskip 0.3 true cm

{\tiny (M. Sababheh) Department of Basic Sciences, Princess Sumaya University for Technology, Amman 11941,
	Jordan. 
	
	\textit{E-mail address:} sababheh@yahoo.com; sababheh@psut.edu.jo}

\newpage
\section*{Appendix: Proofs of results in Section \ref{sec3}.}

{\bf Proof of Proposition \ref{proposition01}:}
We calculate
\begin{equation}\label{prop01_proof_ineq01}
\frac{dK(x,t)}{dx}=\frac{f_t(x)g_t(x)}{(1-t)(1-x^t)(x-1)^2}\left(\frac{t-1}{t}\frac{x^t-1}{x^t-x}\right)^t,
\end{equation}
where
$$
f_t(x):=1-x^t+xt-t,\quad g_t(x):=tx^{t-1}-1+(1-t)x^{t}.
$$
Then we have 
$$\dfrac{df_t(x)}{dx}=t(1-x^{t-1}),\quad \dfrac{dg_t(x)}{dx}=t(1-t)x^{t-2}(x-1).$$ 
\begin{itemize}
\item[(i)] For the case $t<0$ or $t>1$, we have $(1-t)(1-x^t)<0$ for $0<x<1$, and
$(1-t)(1-x^t)>0$ for $x>1$. We also find that $f_t(x)\le f_t(1)=0$ and $g_t(x)\le g_t(1)=0$ for all $x>0$. Thus we have  $\dfrac{dK(x,t)}{dx}\le 0$ for $0<x<1$, and $\dfrac{dK(x,t)}{dx}\ge 0$ for $x>1$.
\item[(ii)] For the case $0<t<1$, we have $(1-t)(1-x^t)>0$ for $0<x<1$, and
$(1-t)(1-x^t)<0$ for $x>1$. We also find that $f_t(x)\ge f_t(1)=0$ and $g_t(x)\ge g_t(1)=0$ for all $x>0$. Thus we have  $\dfrac{dK(x,t)}{dx}\ge 0$ for $0<x<1$, and $\dfrac{dK(x,t)}{dx}\le 0$ for $x>1$.
\end{itemize}
Finally we note that $\lim\limits_{x\to 1} K(x,t)=1$, see \cite[Theorem 2.54 (iii)]{FHPS2005} for example.
\qed \hfill

{\bf Proof of Lemma \ref{lemma01}:}
\begin{itemize}
\item[(I)]
We set the function
$$
f(x,t):=\frac{x+1}{x-1}- (1-t)^2\left(\frac{x+x^t}{x-x^t}\right)-t^2\left(\frac{x^t+1}{x^t-1}\right),\quad (x > 1,\,\,0< t< 1).
$$
Since it can be easily proven that $f(x,1/2)>0$ for $x>1$ and the symmetric property such as $f(x,1-t)=f(x,t)$, we have only to prove $f(x,t)> 0$ for $x > 1$ and $0< t < 1/2$.
Then we calculate
$$
\frac{df(x,t)}{dt} =-\frac{2g(x,t)}{(x-x^t)^2(x^t-1)^2},
$$
where
$$g(x,t):=(x^t-1)(x^t-x)h(x,t)+x^t(\log x) k(x,t)$$ in which 

$$
h(x,t):=(2t-1)x^t(1-x)+x^{2t}-x, 
$$
and 
$$k(x,t):=(t-1)^2x(x^{2t}+1)-t^2(x^2+x^{2t})+2(2t-1)x^{t+1}.$$
\begin{itemize}
\item[(i)]
 We have
$$
\frac{dh(x,t)}{dx}=x^t-1-2t^2(x-1)x^{t-1}+tx^{t-1}(2x^t-x-1),\quad \frac{d^2h(x,t)}{dx^2}=t(1-2t)x^{t-2}l(x,t),
$$
where
$$
l(x,t):=t(x-1)+x+1-2x^t.
$$
Then we have
$$
\frac{dl(x,t)}{dx}=t+1-2tx^{t-1},\quad \frac{d^2l(x,t)}{dx^2}=-2t(t-1)x^{t-2}> 0,\quad (x> 1,\,\,\,0< t < 1/2).
$$
Consequently, $\dfrac{dl(x,t)}{dx} > \dfrac{dl(1,t)}{dx}=1-t> 0$ which implies $l(x,t)> l(1,t)=0$.
Since $1-2t < 0$, this means $\dfrac{d^2h(x,t)}{dx^2}> 0$ so that we have $\dfrac{dh(x,t)}{dx}> \dfrac{dh(1,t)}{dx}=0$ which implies $h(x,t)> h(1,t)=0$  for $x > 1$ and $0< t < 1/2$.
\item[(ii)]
Similarly, we calculate
\begin{eqnarray*}
 \frac{dk(x,t)}{dx}&=&2t^3(x-1)x^{2t-1}+2t(x^t-1)+(x^t-1)^2-t^2(3x^{2t}-4x^t+2x-1),\\
 \frac{d^2k(x,t)}{dx^2}&=&-2t\left\{t+(t+1)(1-2t)x^{t-1}-t^2(1-2t)x^{2t-2}-(t-1)^2(2t+1)x^{2t-1}\right\},\\
 \frac{d^3k(x,t)}{dx^3}&=&2t(2t-1)(t-1)x^{t-3}u(x,t), \;{\text{where}}\\
\end{eqnarray*}
$$u(x,t):=(t+1)x-2t^2x^t+(t-1)(2t+1)x^{t+1}.$$
Then we have 
\begin{eqnarray*}
&&\frac{du(x,t)}{dx}=t+1-2t^3x^{t-1}+(t+1)(t-1)(2t+1)x^t,\\
&&\frac{d^2u(x,t)}{dx^2}=t(t-1)x^{t-2}\left(2t^2(x-1)+3tx+x\right)< 0,\,\,\,\,(x> 1,\,\,\,0< t < 1/2).
\end{eqnarray*}
Thus  $\dfrac{du(x,t)}{dx}<\dfrac{du(1,t)}{dx}=t(t-1)< 0$ which implies $u(x,t)< u(1,t)=0$.
Since $2t-1< 0$, this means $\dfrac{d^3k(x,t)}{dx^3}< 0$ so that $\dfrac{d^2k(x,t)}{dx^2}< \dfrac{d^2k(1,t)}{dx^2}=0$ which implies $\dfrac{dk(x,t)}{dx}< \dfrac{dk(1,t)}{dx}=0$. Therefore $k(x,t)< k(1,t)=0$ for $x > 1$ and $0< t < 1/2$.
\end{itemize}
From (i) and (ii), we have $g(x,t)< 0$, namely we have $\dfrac{df(x,t)}{dt}> 0$ which implies $f(x,t)> \lim\limits_{t \to 0^+}f(x,t)=0$. 

Finally, replacing $x$ by $1/x$ in \eqref{lemma01_ineq01}, we have the reversed inequality of \eqref{lemma01_ineq01} for $0<x<1$ and $0< t < 1$.

\item[(II)]
As we stated in the beginning of (I), we have the symmetric property for $f(x,t)$ such that $f(x,1-t)=f(x,t)$. Therefore it is sufficient to consider the case $t > 1$. 
Putting  $t:=1/s< 1$ in \eqref{lemma01_ineq01}, we have
$$
\frac{x+1}{x-1} > \left(1-\frac{1}{s}\right)^2\left(\frac{x+x^{1/s}}{x-x^{1/s}}\right)+\frac{1}{s^2}\left(\frac{x^{1/s}+1}{x^{1/s}-1}\right),\quad (x>1,\,\,\,s> 1).
$$
Multiplying $s^2>0$ to both sides and putting $y:=x^{1/s}>1$, we have
$$
\frac{y+1}{y-1}< (1-s)^2\left(\frac{y+y^s}{y-y^s}\right)+s^2\left(\frac{y^s+1}{y^s-1}\right),\quad (y>1,\,\,\,s> 1),
$$
which shows  the inequality of \eqref{lemma01_ineq01r}.

Finally, replacing $x$ by $1/x$ in \eqref{lemma01_ineq01r}, we have the reversed inequality of \eqref{lemma01_ineq01r} for $0<x<1$.
\end{itemize}
\qed \hfill

{\bf Proof of Theorem \ref{theorem01}:}
The equalities $L(x,0)=1=K(x,0)$ and $L(x,1)=1=K(x,1)$ are trivial. See \cite[Theorem 2.54]{FHPS2005} for example. That is, both equalities in \eqref{theorem01_ineq01} and \eqref{theorem01_ineq02} holds when $t=0$ or $1$. In the following, we prove the strict inequalities of \eqref{theorem01_ineq01} and \eqref{theorem01_ineq02} for the cases $t\neq \{0,1\}$. 
\begin{itemize}
\item[(I)]
The equality holds for the special case $x=1$, since $K(1,t)=\lim\limits_{x\to 1} K(x,t)=1$ for $t\in \mathbb{R}$, by l'Hospital's theorem. We firstly assume $x>1$. By elementary calculations, we have $L(1/x,t)=L(x,t)$ and  it is known \cite[Theorem 2.54 (i)]{FHPS2005} that $K(1/x)=K(x,t)$ for all $t \in \mathbb{R}$ and $x>0$.  Therefore it is sufficient to prove \eqref{theorem01_ineq01} for $0<t < 1$ and $x>1$. 
The inequality \eqref{theorem01_ineq01} without an equality case is equivalent to the following inequality
\begin{equation}\label{theorem01_ineq02}
\frac{(1-t)(x-1)(x+x^t)}{(x+1)(x-x^t)}< \left(\frac{(1-t)(x^t-1)(x+x^t)}{t(x-x^t)(x^t+1)}\right)^t.
\end{equation}
From the ordering of the weighted means, we generally have the inequality $a^t> \dfrac{a}{(1-t)a+t}$ for $a>0$ and $0< t < 1$. In order to prove the inequality \eqref{theorem01_ineq02}, it is sufficient to prove
 \begin{equation}\label{theorem01_ineq03}
 \frac{(1-t)(x-1)(x+x^t)}{(x+1)(x-x^t)}< \frac{a_{x,t}}{(1-t)a_{x,t}+t},\quad a_{x,t}:=\frac{(1-t)(x^t-1)(x+x^t)}{t(x-x^t)(x^t+1)}>0.
 \end{equation}
By elementary calculations with $1< x^t < x$ for $x > 1$ and $0< t < 1$, the inequality given in \eqref{theorem01_ineq03} is equivalent to
 \begin{eqnarray*}
&&  \frac{(1-t)(x-1)(x+x^t)}{(x+1)(x-x^t)}< \frac{(1-t)(x^t-1)(x^t+x)}{(1-t)^2(x^t-1)(x^t+x)+t^2(x-x^t)(x^t+1)}\\
&&   \Leftrightarrow (x+1)(x^t-1)(x-x^t)> (x-1)\left\{ (1-t)^2(x^t-1)(x^t+x)+t^2(x-x^t)(x^t+1)\right\}\\
&&\Leftrightarrow \frac{x+1}{x-1}> (1-t)^2\left(\frac{x+x^t}{x-x^t}\right)+t^2\left(\frac{x^t+1}{x^t-1}\right),
 \end{eqnarray*}
 which is true for $x > 1$ and $0< t < 1$ thanks to the inequality \eqref{lemma01_ineq01} in Lemma \ref{lemma01}.
 \item[(II)]
 We prove the inequality \eqref{theorem01_ineq01r}. 
 By elementary calculations, we have $L(x,1-t)=L(x,t)$  for all $t \in \mathbb{R}$ and $x>0$.  It is also known \cite[Theorem 2.54 (ii)]{FHPS2005} that  $K(x,1-t)=K(x,t)$ for all $t \in \mathbb{R}$ and $x>0$. 
 Therefore it is sufficient to prove \eqref{theorem01_ineq01r} for $t > 1$ and $x>0$. 
Putting   $t:=1/s< 1$  in \eqref{theorem01_ineq01} without an equality case, we have
$$
\left(\frac{x^{1/s}+x}{x+1}\right)\left(\frac{x^{1/s}+1}{x^{1/s}+x}\right)^{1/s}< \frac{\left(s^{1/s}-x\right)}{(1/s-1)(x-1)}\left(\frac{(1/s-1)\left(x^{1/s}-1\right)}{1/s\left(x^{1/s}-x\right)}\right)^{1/s},\quad (x>0,\,\,\,s> 1).
$$
Taking $s$--th power of the both sides and then taking the inverse  of the both sides, we have
$$
\left(\frac{x^{1/s}+x}{x^{1/s}+1}\right)\left(\frac{x+1}{x^{1/s}+x}\right)^s> \frac{\left(x^{1/s}-x\right)}{(1-s)\left(x^{1/s}-1\right)}\left(\frac{(1-s)\left(x-1\right)}{s\left(x^{1/s}-x\right)}\right)^s,\quad (x>0,\,\,\,s> 1).
$$
 Putting $y:=x^{1/s}>0$, we have
 $$
 \left(\frac{y^s+y}{y+1}\right) \left(\frac{y^s+1}{y^s+y}\right)^s> \frac{ \left(y-y^s\right)}{(1-s)(y-1)} \left(\frac{(1-s) \left(y^s-1\right)}{s \left(y-y^s\right)}\right)^s,\quad (y>0,\,\,\,s> 1),
 $$
 which shows the inequality \eqref{theorem01_ineq01r}.
  \end{itemize}
  This completes the proof.
\qed \hfill

{\bf Proof of Corollary \ref{cor01}:}
\begin{itemize}
\item[(I)]  By elementary calculations, we see the inequality \eqref{corollary01_ineq01} is equivalent to the inequality:
\begin{eqnarray*}
&& \frac{(x^{2t}-x^2)}{(t-1)(x^2-1)}\left(\frac{t-1}{t}\frac{x^{2t}-1}{x^{2t}-x^2}\right)^t\le \frac{(x^t-x)^2}{(t-1)^2(x-1)^2}\left(\frac{(t-1)^2}{t^2}\frac{(x^t-1)^2}{(x^t-x)^2}\right)^t \\
&& \Leftrightarrow \frac{(1-t)(x-1)(x+x^t)}{(x+1)(x-x^t)}\le \left(\frac{1-t}{t}\frac{(x^t-1)(x+x^t)}{(x-x^t)(x^t+1)}\right)^t
\end{eqnarray*}
which is  equivalent to  \eqref{theorem01_ineq01}. 

\item[(II)] Since $K(x,1-t)=K(x,t)$ for all $t\in\mathbb{R}$, we have only to prove the inequality \eqref{corollary02_ineq01} for $t\ge 1$.
From (I), we have the inequality \eqref{corollary01_ineq01} for  $x > 0$ and $0\le t \le 1$. As stated in the proof of (I),  the inequality \eqref{corollary01_ineq01} is equivalent to the inequality \eqref{theorem01_ineq01}. In the inequality \eqref{theorem01_ineq01}, we put $t:=1/s$ for $s \ge 1$ and $x^{1/s}:=y > 0$. Then the inequality \eqref{theorem01_ineq01} is written as
$$
\left(\frac{y^s+y}{y^s+1}\right)\left(\frac{y+1}{y^s+y}\right)^{1/s}\le \frac{s(y^s-y)}{(s-1)(y^s-1)}\left(\frac{(s-1)(y-1)}{(y^s-y)}\right)^{1/s}
$$
by elementary calculations. Taking $s$--th power of the both sides and then taking the inverse  of the both sides, we have, 
$$
\left(\frac{y^s+y}{y+1}\right)\left(\frac{y^s+1}{y^s+y}\right)^s\ge \frac{(y^s-y)}{(s-1)(y-1)}\left(\frac{(s-1)(y^s-1)}{s(y^s-y)}\right)^s.
$$
which is equivalent to the inequality $K(y^2,s)\ge K^2(y,s)$  for $s\ge 1$ and $y > 0$, by elementary calculations. 
Thus the inequality \eqref{corollary02_ineq01} is true for $x > 0$, under the assumption $t\ge 1$ or $t \le 0$.
\end{itemize}
\qed \hfill

{\bf Proof of Proposition \ref{last_prop}:}
Corollary \ref{cor_spec_less_geo} states
$$
A{{\natural }_{t}}B\le \Gamma(m,M,t)A{{\sharp }_{t}}B.
$$
Now replace $A$ and $B$ by $A^{-1}$ and $B^{-1}$, respectively. This will lead to replacing $m$ and $M$ by $\frac{1}{M}$ and $\frac{1}{m}$, respectively. Noting that $\Gamma(m,M,t)=\Gamma\left(\frac{1}{M},\frac{1}{m},t\right)$, we find that
\[A^{-1}{{\natural }_{t}}B^{-1}\le \Gamma(m,M,t)A^{-1}{{\sharp }_{t}}B^{-1},\]
which implies
\[\left(A^{-1}{{\natural }_{t}}B^{-1}\right)^{-1}\geq \Gamma(m,M,t)^{-1}\left(A^{-1}{{\sharp }_{t}}B^{-1}\right)^{-1}.\]
This is equivalent to
\begin{equation}\label{eq_remark}
A\sharp_tB\leq \Gamma(m,M,t)A\natural_t B.
\end{equation}
Corollary \ref{cor_geo_less_spec} also states 
\begin{equation}\label{eq_remark00}
A{{\sharp }_{t}}B\le \eta(m,M,t)A{{\natural }_{t}}B.
\end{equation}
In the sequel, we show $\eta(m,M,t) \le  \Gamma(m,M,t)$. 
We firstly show 
\begin{equation}\label{needed_rem_1}
{{\left( \frac{{{\left( M+m \right)}^{2}}}{4Mm} \right)}^{1+t}}\leq \frac{{{\left( {{M}^{1+t}}+{{m}^{1+t}} \right)}^{2}}}{4{{M}^{1+t}}{{m}^{1+t}}}.
\end{equation}
Since the equality holds in \eqref{needed_rem_1} when $t=0$, we assume $0<t \le 1$.
The H\"older inequality for $a_1,a_2,b_1,b_2>0$ with $\dfrac{1}{p}+\dfrac{1}{q}=1$ and $p,q>1$ states that $a_1b_1+a_2b_2\le\left(a_1^p+a_2^p\right)^{1/p}\left(b_1^q+b_2^q\right)^{1/q}$. Putting $a_1:=1,\,\,a_2:=1$ , $b_1:=M,\,\,b_2:=m$ and $p:=\dfrac{t+1}{t}>1,\,\,q:=t+1>1$, we have
$$
1\cdot M+1\cdot m \le \left(1^{\frac{t+1}{t}}+1^{\frac{t+1}{t}}\right)^{\frac{t}{t+1}}\left(M^{t+1}+m^{t+1}\right)^{\frac{1}{t+1}}
$$
 Taking $(t+1)$--th power of the both sides, we have $(M+m)^{t+1}\le 2^{t}\left(M^{t+1}+m^{t+1}\right)$
which is equivalent to \eqref{needed_rem_1}. Thus we have the inequality \eqref{needed_rem_1} for $0<m<M$ and  $0\le t \le 1$.
We secondly prove
\begin{align}\label{needed_rem_2}
K\left( \frac{m}{M},\frac{M}{m},t \right)\le {{K}^{2}}\left( \sqrt{\frac{m}{M}},\sqrt{\frac{M}{m}},t \right).
\end{align}
Putting $x:=\sqrt{\dfrac{M}{m}}>1$ for $0<m<M$, we have the relation with  the different symbol given in \eqref{def_K02}:
$$
K\left(\frac{m}{M},\frac{M}{m},t\right)=K\left(\frac{1}{x},x,t\right)=K(x^2,t),\quad 
 K\left(\sqrt{\frac{m}{M}},\sqrt{\frac{M}{m}},t\right)=K\left(\sqrt{\frac{1}{x}},\sqrt{x},t\right)=K(x,t).
$$
With these and the obtained inequality \eqref{corollary01_ineq01} for $0\le t \le 1$, we have the  inequality \eqref{needed_rem_2} for $0<m<M$ and $0\le t \le 1$. 
The inequality \eqref{needed_rem_2} together with \eqref{needed_rem_1} implies that $\eta(m,M,t)\leq \Gamma(m,M,t)$.
\qed \hfill

\end{document}